\documentclass{amsart}
\usepackage[all]{xy}
\usepackage[dvipsnames]{xcolor}
\usepackage{amssymb}
\usepackage{enumerate}
\usepackage{mathabx}
\usepackage{tikz,calligra,mathrsfs}
\usetikzlibrary{matrix,arrows,decorations.pathmorphing}
\usepackage[all]{xy}
\usepackage{graphicx}
\usepackage{caption}
\usepackage{cleveref}
\usetikzlibrary{positioning}
\let\cal\mathcal

\def\Bscr{{\cal B}}

\def\Fscr{{\cal F}}
\def\Gscr{{\cal G}}

\def\Kscr{{\cal K}}
\def\Lscr{{\cal L}}
\def\Mscr{{\cal M}}

\def\Oscr{{\cal O}}

\def\Tscr{{\cal T}}

\def\MultiProj{\operatorname{(Multi)Proj}}

\def\Bimod{\operatorname{Bimod}}

\def\Gr{\operatorname{Gr}}

\def\QGr{\operatorname{QGr}}

\def\Qcoh{\operatorname{Qcoh}}

\def\Hom{\operatorname {Hom}}

\def\im{\operatorname {im}}
\def\coker{\operatorname {coker}}

\def\Ker{\operatorname {ker}}

\def\Pic{\operatorname {Pic}}

\DeclareMathOperator{\Proj}{Proj}

\DeclareMathOperator{\Tors}{Tors}

\DeclareMathOperator{\Aut}{Aut}

\def\Frac{\operatorname{Frac}}

\newtheorem{lemma}{Lemma}[section]
\newtheorem{proposition}[lemma]{Proposition}
\newtheorem{theorem}[lemma]{Theorem}
\newtheorem{corollary}[lemma]{Corollary}

\newtheorem*{theorem*}{Theorem}

\theoremstyle{definition}

\newtheorem{definition}[lemma]{Definition}

{

}

\theoremstyle{remark}

\newtheorem{remark}[lemma]{Remark}

\newtheorem{notation}[lemma]{Notation}

\newdimen\uboxsep \uboxsep=1ex
\def\uboxn#1{\vtop to 0pt{\hrule height 0pt depth 0pt\vskip\uboxsep
\hbox to 0pt{\hss #1\hss}\vss}}

\def\uboxs#1{\vbox to 0pt{\vss\hbox to 0pt{\hss #1\hss}
\vskip\uboxsep\hrule height 0pt depth 0pt}}

\let\oldmarginpar\marginpar
\def\marginpar#1{\oldmarginpar{\tiny\textcolor{red}{#1}}}

\author{Dennis Presotto}
\title{$\mathbb{Z}^2$-algebras as noncommutative blow-ups.}

\begin{document}
\begin{abstract}
The goal of this note is to first prove that for a well behaved $\mathbb{Z}^2$-algebra $R$, the category $\QGr(R):=\Gr(R)/\Tors(R)$ is equivalent to $\QGr(R_\Delta)$ where $R_\Delta$ is a \emph{diagonal-like} sub-$\mathbb{Z}$-algebra. Afterwards we use this result to prove that the $\mathbb{Z}^2$-algebras as introduced in \cite{Presotto} are $\QGr$-equivalent to a diagonal-like sub-$\mathbb{Z}$-algebra which is a simultaneous noncommutative blowup of a quadratic and a cubic Sklyanin algebra. As such we link the noncommutative birational transformation and the associated $\mathbb{Z}^2$-algebras in \cite{PresVdB} and \cite{Presotto} with the noncommutative blowups as in \cite{RSS2} and \cite{VdB19}.
\end{abstract}
\maketitle
\tableofcontents

\section{Introduction}
Throughout this note $R$ is a $\mathbb{Z}^2$-algebra over some algebraically closed field $k$. (See \S \ref{sec:diagonal-like} for the definition of $\mathbb{Z}^2$-algebras. We refer the interested reader to \cite{Sierra} for a more thorough introduction to $G$-algebras). Following tradition \cite{artinzhang} we associate a noncommutative projective scheme $\Proj(R)$ to a $\mathbb{Z}^2$-algebra $R$ whenever $R$ is sufficiently well behaved (e.g. a Noetherian $\mathbb{Z}^2$-algebra satisfying some analogue of generation in degree 1, see \Cref{def:gendegree1}). $\Proj(R)$ is defined via its category of ``quasicoherent sheaves'':
\[ \Qcoh( \Proj(R)) := \QGr(R) = \Gr(R)/ \Tors(R) \]
where $\Gr(R)$ is the category of graded right $R$-modules and $\Tors(R)$ is the full subcategory of torsion $R$-modules
. In this way these $\mathbb{Z}^2$-algebras give non-commutative generalizations of bi-homogeneous algebras.\\

Multi-homogeneous algebras (i.e. $\mathbb{Z}^n$-graded algebras with with $n>1$) appear frequently in the literature \cite{chan, CHTV, KSSW, SimisTV, trung}.
These multi-homogeneous algebras $S$ inherit many properties from diagonal subalgebras $S_\Delta$. Moreover to each multi-homogeneous algebra $S$ one associates a projective scheme $\MultiProj(S)$ and given suitable conditions on $S$ this projective scheme coincides with $\MultiProj(S_\Delta)$ (see for example \cite[Lemma 1.3]{trung}). For a bi-homogeneous algebra such a diagonal subalgebra is simply a $\mathbb{Z}$-graded algebra and as such $\MultiProj(S)$ is isomorphic to the $\Proj$ of a graded algebra. In \S  \ref{sec:diagonal-like} we generalize this result to the level of $\mathbb{Z}^2$-algebras and prove (\Cref{thm:equivalencediagonal}) that for sufficiently well behaved $\mathbb{Z}^2$-algebras $R$ the category $\QGr(R)$ is equivalent to $\QGr(R_\Delta)$ where $R_\Delta$ is a \emph{diagonal-like} sub-$\mathbb{Z}$-algebra. Using more abstract techniques, a similar result was obtained independently by Lowen, Ramos-Gonz\'alez and Shoikhet in \cite{LowenRS}.\\

Our main application of \Cref{thm:equivalencediagonal} is to the $\mathbb{Z}^2$-algebras appearing in \cite{Presotto}. These $\mathbb{Z}^2$-algebras were constructed when investigating the noncommutative versions of the standard birational transformation $\mathbb{P}^1 \times \mathbb{P}^1 \dashrightarrow \mathbb{P}^2$ as in \cite{PresVdB}. Such a noncommutative $\mathbb{P}^1 \times \mathbb{P}^1 \dashrightarrow \mathbb{P}^2$ is obtained from an inclusion of $\mathbb{Z}$-algebras
\[ \check{A'} \hookrightarrow \check{A}^{(2)} \]
where $A'$ a quadratic Sklyanin algebra, $A$ is a cubic Sklyanin algebra, $\check{A'}$ and $\check{A}$ are their associated $\mathbb{Z}$-algebras and $\check{A}^{(2)}$ is the second Veronese algebra of $\check{A}$ (i.e. $\check{A}^{(2)}_{i,j} = \check{A}_{2i,2j}:= A_{2j-2i}$).

Recall from \cite{ATV1} that quadratic and cubic Sklyanin algebras are classified using triples of geometric data $(Y,\Lscr,\sigma)$ where $Y$ is a smooth elliptic curve, $\Lscr$ is a line bundle on $Y$ and $\sigma \in \Aut(Y)$. It was shown in \cite{ATV2} that there is a 1-1-correspondence between points of $Y$ and point modules of $A(Y,\Lscr,\sigma)$. The inclusion $\check{A'} \hookrightarrow \check{A}^{(2)}$ is compatible with these geometric data in the sense that $Y = Y'$ and that the inclusion is constructed starting from some point $p \in Y$.

In \cite{Presotto} it was shown that there is an inclusion (constructed starting from two points $p', q'$)
\[ \check{A} \hookrightarrow \check{A'} \]
such that the composition $\check{A} \hookrightarrow \check{A}^{(2)}$ and $\check{A'} \hookrightarrow \check{A}^{\prime (2)}$ induce the identity on the associated function fields $\Frac(A)_0$ and $\Frac(A')_0$. This was done by investigating a $\mathbb{Z}^2$-algebra $\tilde{A}$ ``containing'' both $A$ and $A'$ as respectively a column and row. In \S \ref{sec:Zsquarenoeth} and \S \ref{sec:Btildenoeth} we check that $\tilde{A}$ satisfies the condition of \Cref{thm:equivalencediagonal}. As such $\QGr(\tilde{A}) \cong \QGr(\tilde{A}_\Delta)$ for each diagonal-like $\mathbb{Z}$-algebra $\tilde{A}_\Delta$.

Finally in \S \ref{sec:asblowups} we focus on a specific $\Delta$. For this $\Delta$ the methods in \cite{Presotto} provide us with inclusions 
\begin{equation} \label{eq:twoinclusions} \tilde{A}_\Delta \hookrightarrow \check{A}^{(4)} \textrm{ and } \tilde{A}_\Delta \hookrightarrow \check{A}^{\prime (3)} \end{equation}
We show that these inclusions are simultaneously compatible with the 1-periodicity of $\check{A}^{(4)}$ and $\check{A}^{\prime (3)}$. As such there is a graded algebra $T$ for which $\check{T} \cong \tilde{A}_\Delta$ and the inclusions in \eqref{eq:twoinclusions} give rise to inclusions $T \hookrightarrow A^{(4)}$ and $T \hookrightarrow A^{\prime (3)}$. Finally we check that these inclusions give $T$ the construction of a noncommutative blowup $A^{(4)}(p)$ and $A^{\prime(3)}(p'+q')$ as in \cite{RSS2}, resulting in our main result:

\begin{theorem} \label{thm:commonblowup}
Let $A$, $A'$, $p$, $p',q'$ and $\tilde{A}$ be as above, then $A$ and $A'$ contain a common blowup $A^{(4)}(p) \cong T \cong A^{\prime(3)}(p'+q')$ and there is an equivalence of categories
\[ \QGr(\tilde{A}) \cong \QGr(T) \]
\end{theorem}
\begin{remark}
With the exception of the equivalence of categories the above result was announced independently by Rogalski, Sierra and Stafford in \cite[Theorem 1.7]{RSS3}.
\end{remark}
\begin{remark}
 In \cite{PresVdB} one considers a noncommutative version of the Cremona transform $\mathbb{P}^2 \dashrightarrow \mathbb{P}^2$. Such a noncommutative Cremona transform is given by an inclusion $\check{A'} \hookrightarrow \check{A}^{(2)}$ where $A$ and $A'$ are quadratic Sklyanin algebras. This inclusion is constructed with respect to some divisor $d$. The methods in \cite{Presotto} provide an ``inverse'' inclusion $\check{A} \hookrightarrow \check{A'}^{(2)}$ with respect to some explicit divisor $d'$. Moreover analogously to op.cit one can construct a $\mathbb{Z}^2$-algebra $\tilde{A}$ containing both $A$ and $A'$. It is a straightforward check that the results in \S \ref{sec:Zsquarenoeth} and \S \ref{sec:Btildenoeth} are also applicable to these $\mathbb{Z}^2$-algebras. If one choses $\Delta(i)=(i,i)$ computations analogous to the ones in \S \ref{sec:asblowups} show that $A$ and $A'$ contain a common blowup $A^{(3)}(d) \cong T \cong A^{\prime (3)}(d')$ and there is an equivalence of categories
\[ \QGr(\tilde{A}) \cong \QGr(T) \]
In order to avoid unnecessary difficult notations, we have chosen to only focus on the proof of \Cref{thm:commonblowup}.
\end{remark}

\section{Acknowledgements}
The author wishes to thank Shinnosuke Okawa for mentioning the idea that geometrically (at the level of $\QGr$) the $\mathbb{Z}^2$-algebras as in \cite{Presotto} are the same as the blowups in \cite{RSS2}. The author wishes to thank Wendy Lowen and Julia Ramos Gonz\'alez for useful discussions about the more abstract nature of the equivalence of categories obtained in \Cref{thm:equivalencediagonal}.
The author further wishes to thank Michel Van den Bergh for providing insights in the most technical aspects of this paper as well as for reading through earlier versions.

\section{Diagonal-like subalgebras}
\label{sec:diagonal-like}

Throughout this section $R$ is a $\mathbb{Z}^2$-algebra over some field $k$. I.e. $R$ is a $k$-algebra together with a decomposition
\[ R = \bigoplus_{(i,j),(m,n) \in \mathbb{Z}^2} R_{(i,j),(m,n)} \]
such that addition is degree-wise and multiplication satisfies 
\[ R_{(a,b),(i,j)} R_{(i,j),(m,n)} \subset R_{(a,b),(m,n)} \textrm{ and } R_{(a,b),(c,d)}R_{(i,j),(m,n)} =0 \textrm{ whenever } (c,d) \neq (i,j)\]
 Moreover there are local units $e_{(i,j)} \in R_{(i,j),(i,j)}$ such that for each $x \in R_{(a,b),(m,n)}:$ 
\[ e_{(a,b)}x = x =x e_{(m,n)}\]
 A graded $R$-module is an $R$-module $M$ together with a decomposition 
\[ M = \bigoplus_{(i,j) \in \mathbb{Z}^2} M_{(i,j)} \]
such that the $R$-action on $M$ satisfies 
\[ M_{(i,j)} R_{(i,j),(m,n)} \subset M_{(m,n)} \textrm{ and } M_{(i,j)}R_{(a,b),(m,n)} = 0 \textrm{ if } (a,b) \neq (i,j) \]
We denote $\Gr(M)$ for the category of graded $R$-modules. In this section we also make the assumption that $R$ is Noetherian in the sense that $\Gr(R)$ is locally Noetherian, moreover we assume that each $R_{(i,j),(m,n)}$ is a finite dimensional vectorspace and $R_{(i,j),(i,j)} = k$.
\begin{notation} \label{not:firstquadrant}
Let $R$ be a $\mathbb{Z}^2$-algebra. Then we denote $R_{+}$ for the $\mathbb{Z}^2$-subalgebra:
\[ \left( R_{+} \right) _{(i,j),(m,n)} = \begin{cases} R_{(i,j),(m,n)} & \textrm{ if $i \leq m$ and $j \leq n$} \\ 0 & \textrm{else} \end{cases} \]
\end{notation}

\begin{definition} \label{def:gendegree1}
Let $R$ be a $\mathbb{Z}^2$-algebra and let $R_{+}$ be as above. We say that $R_{+}$ is \emph{generated in degree $(0,1)$ and $(1,0)$} if each homogeneous element in $R_+$ can be written as a product of elements of degree $(0,1)$ or $(1,0)$. I.e.:
\begin{eqnarray*}
& \forall i,j,m,n \in \mathbb{Z}, i< m, j \leq n:  R_{(i,j),(i+1,j)} \otimes R_{(i+1,j),(m,n)} \rightarrow R_{(i,j),(m,n)} \textrm{ is surjective} & \\
& \textrm{ and } & \\
& \forall i,j,m,n \in \mathbb{Z}, i \leq m, j < n: R_{(i,j),(i,j+1)} \otimes R_{(i,j+1),(m,n)} \rightarrow R_{(i,j),(m,n)}\textrm{ is surjective} &
\end{eqnarray*}
\end{definition}

\begin{definition} \label{def:Z2torsion}
Let $R$ be a $\mathbb{Z}^2$-algebra such that $R_{+}$ is generated in degree $(0,1)$ and $(1,0)$ and let $M$ be a graded $R$-module. We say $M$ is \emph{right-upper-bounded} if there exist $i_0, j_0 \in \mathbb{Z}: \forall i \geq i_0, j \geq j_0: M_{(i,j)} = 0$. $M$ is said to be \emph{torsion} if it is a direct limit of right-upper-bounded modules. We denote $\Tors(R)$ for the full subcategory of torsion modules in $\Gr(R)$.
\end{definition}
The assumption that $R$ is Noetherian implies that $\Tors(R)$ is a Serre subcategory of $\Gr(R)$ and as such we can define a quotient category
\[ \QGr(R) := \Gr(R) / \Tors(R) \]
The main result of this section is that we can understand $\QGr(R)$ in terms of \emph{diagonal-like} sub-$\mathbb{Z}$-algebras of $R$.
\begin{definition}
Let $\Delta: \mathbb{Z} \rightarrow \mathbb{Z}^2$ and let $R$ be a $\mathbb{Z}^2$-algebra, then we denote $R_\Delta$ the induced $\mathbb{Z}$-subalgebra:
\[ \left( R_\Delta \right)_{i,j} := R_{\Delta(i), \Delta(j)} \]
We say $\Delta$ (or $R_\Delta$) is \emph{diagonal-like} if for all $(a,b) \in \mathbb{Z}^2$ there is an $i \in \mathbb{Z}$ such that $(a,b) \leq \Delta(i)$ with $\leq$ the product order on $\mathbb{Z}^2$.
\end{definition}

\begin{theorem} \label{thm:equivalencediagonal}
Let $R$ be a $\mathbb{Z}^2$-algebra and let $R_\Delta$ be a diagonal-like sub-$\mathbb{Z}$-algebra. Then there is an equivalence of categories
\begin{equation}
\label{eq:equivalentQGr}
\QGr(R) \cong \QGr(R_\Delta)
\end{equation}
\end{theorem}
\begin{proof}
There is a restriction functor $F: \Gr(R) \rightarrow \Gr(R_\Delta)$ defined by $F(M)_i := M_{\Delta(i)}$. $F$ obviously maps torsion modules to torsion modules and hence induces a functor $F: \QGr(R) \rightarrow \QGr(R_\Delta)$. 

Next we define a right exact functor $G: \Gr(R_\Delta) \rightarrow \Gr(R)$ by setting $G(e_i R_\Delta)~=~e_{\Delta(i)}R$ at the level of objects. As $\Hom_{R_\Delta}(e_i R_\Delta, e_j R_\Delta)$ and $\Hom_R(e_{\Delta(i)}R,e_{\Delta(j)}R)$ are both canonically isomorphic to $R_{\Delta(j),\Delta(i)}$, $G: \Hom_{R_\Delta}(e_i R_\Delta, e_j R_\Delta) \rightarrow \Hom_R(e_{\Delta(i)}R,e_{\Delta(j)}R)$ is chosen to be the identity at the level of  homomorphisms. We now claim that $G$ sends torsion modules to torsion modules. As $G$ is right exact and commutes with direct sums, it is compatible with direct limits. As such, it suffices to check that $G$ sends finitely generated, right bounded $R_\Delta$-modules to right-upper-bounded $R$-modules. For this let $M$ be a finitely generated, right bounded $R_\Delta$-module. There is a resolution
\[ \bigoplus_m e_{i_m} R_\Delta \overset{f}{\rightarrow} \bigoplus_{n=0}^{n_0} e_{j_n} R_\Delta \rightarrow M \rightarrow 0 \]
and there is a $u_0 \in \mathbb{Z}$ such that $f$ is surjective in all degrees $u$ with $u \geq u_0$. Moreover we can assume $u_0 \geq j_n$ for all $n$. Now write $\Delta(u_0)=(a_0,b_0)$. The fact that $R_+$ is assumed to be generated in degree $(0,1)$ and $(1,0)$ implies that the induced map
\[ G(f): \bigoplus_m e_{\Delta(i_m)} R \rightarrow \bigoplus_{n=0}^{n_0} e_{\Delta(j_n)} R \]
is surjective in all degrees $(a,b)$ with $a \geq a_0, b \geq b_0$. This implies $G(M)_{(a,b)}=0$ for all such $a,b$, hence $G(M)$ is right-upper-bounded. In particular the functor $G: \Gr(R_\Delta) \rightarrow \Gr(R)$ induces a functor $G: \QGr(R_\Delta) \rightarrow \QGr(R)$.

It is immediate that $G \circ F = Id$. By \Cref{lem:smallergenset} below it now suffices to check that $F(G(e_{\Delta(i)}R))=e_{\Delta(i)}(R)$. This is however obvious.

\end{proof}
\begin{lemma} \label{lem:smallergenset}
The collection $\{ e_{\Delta(i)}R \mid i \in \mathbb{Z} \}$ (or rather the set of corresponding objects in $\QGr(R)$) forms a set of generators for $\QGr(R)$.
\end{lemma}
\begin{proof}
As the collection $\{ e_{(m,n)}R \mid m,n \in \mathbb{Z} \}$ forms a set of generators for $\Gr(R)$ and hence also for $\QGr(R)$ it suffices to show that every $e_{(m,n)}R$ is a quotient of a direct sum of objects $e_{\Delta(i)}R$ in $\QGr(R)$. For this fix some $m,n \in \mathbb{Z}$. We claim that there are surjective maps (in $\QGr(R)$)
\begin{eqnarray}
\label{eq:firstmap} &  e_{(m+1,n)}R^{\oplus N} \rightarrow e_{(m,n)}R & \\
\notag & \textrm{and} & \\
\label{eq:secondmap} &  e_{(m,n+1)}R^{\oplus N'} \rightarrow e_{(m,n)}R & 
\end{eqnarray}
Assume these claims. As we $\Delta$ was assumed to be diagonal-like, there exist integers $a,b,i \in \mathbb{Z}$ such that $a,b \geq 0$ and $\Delta(i) = (m+a,n+b)$. In particular the surjective maps \eqref{eq:firstmap} and \eqref{eq:secondmap} give rise to a surjective map $e_{\Delta(i)} R^{\oplus N''} \rightarrow e_{(m,n)}R$. Hence the lemma follows from the claims. 

We now prove the claims. As both claims are similar we only prove \eqref{eq:firstmap}. For this let $N = \dim_k(\Hom(e_{m+1,n}R,e_{m,n}R)) = \dim_k ( R_{(m,n),(m+1,n)})$. As we assumed $R$ to be positively generated in degree (1,0) and (0,1) there is a map
\[  e_{(m+1,n)}R^{\oplus N} \rightarrow e_{(m,n)}R \]
in $\Gr(R)$ whose cokernel lives in degrees $(x,y)$ with either $x \leq m$ or $y < n$. As such this cokernel is torsion and the induced map in $\QGr(R)$ is surjective.
\end{proof}


\section{Summary of the constructions in \cite{PresVdB} and \cite{Presotto}}
\label{sec:summarycon}
Throughout the following sections $k$ is assumed to be an algebraically closed field of characteristic zero.\\
In this section we briefly recall the constructions in \cite{PresVdB} and \cite{Presotto} necessary to understand the proof of \Cref{thm:commonblowup}. We refer to these papers for all unexplained notations. Three dimensional Artin-Schelter regular algebras generated in degree 1 provide important examples of noncommutative surfaces. They were defined in \cite{artinschelter} and it was shown that they are either generated by 3 elements satisfying 3 relations of degree 2 (the quadratic case) or 2 generators satisfying 2 relations of degree 3 (the cubic case). For use below we define $(r,s)$ to be respectively the number of generators and the degrees of the relations. Thus $(r,s) = (3,2)$ or $(2,3)$ depending on whether the algebra is quadratic or cubic.\\

Three dimensional Artin-Schelter regular algebras were classified in \cite{ATV1} in terms of geometric triples $(Y,\Lscr,\sigma)$. Our main focus lies on quadratic and cubic ``Sklyanin algebras'' of infinite order. In this case $Y$ is a smooth elliptic curve, $\Lscr$ is a line bundle of degree $s$ and $\sigma: Y \rightarrow Y$ is an infinite order morphism given by a translation. It is customary to write $\tau := \sigma^{s+1}$. Moreover for each Sklyanin algebra $A = A(Y,\Lscr,\sigma)$ there exists a central element $g \in A_{s+1}$ such that $A/gA \cong B(Y,\Lscr,\sigma)$ where $B(Y,\Lscr,\sigma)$ is the twisted homogeneous coordinate ring (\cite{AV}) with respect to $(Y,\Lscr,\sigma)$. \\

In \cite{PresVdB} a noncommutative version of the classical birational transformation $\mathbb{P}^1 \times \mathbb{P}^1 \dashrightarrow \mathbb{P}^2$ was constructed as an inclusion 
\begin{equation}
\label{eq:NCbirational1} \gamma: \check{A'} \hookrightarrow \check{A}^{(2)} 
\end{equation}
where $A=A(Y,\Lscr,\sigma)$ is a cubic Sklyanin algebra, $A'$ a quadratic Sklyanin algebra and $\check{A}$ and $\check{A'}$ are their associated $\mathbb{Z}$-algebras. (i.e. $\check{A}_{i,j} := A_{j-i}$ and multiplication in $\check{A}$ is defined in the obvious way). The inclusion in \eqref{eq:NCbirational1} is constructed starting from a point $p \in Y$. More concretely: one defines $X = \QGr(A)$ and introduces a category of bimodules $\Bimod(X-X)$. Of particular importance are the bimodules $o_X(1)$ (corresponding to degree shifting in $A$) and $m_p$ (the ideal ``sheaf'' of the point $p$). We refer the interested reader to \cite[Chapter 3]{VdB19} for a thorough introduction to bimodules. The inclusion \eqref{eq:NCbirational1} is then obtained by the following identifications
\begin{eqnarray}
\label{eq:trivial} \check{A} &  \cong & \bigoplus_{m,n \in \mathbb{Z}}\Hom_X(\Oscr_X(-n), \Oscr_X(-m)) \\
\label{eq:nontrivial} \check{A'} & \cong & \bigoplus_{m,n \in \mathbb{Z}}\Hom_X(\Oscr_X(-2n), \Oscr_X(-2m) \otimes m_{\tau^{-m}p} \ldots m_{\tau^{-n+1}p} )
\end{eqnarray}
The identification in \eqref{eq:nontrivial} is obtained by noticing that the right hand side is generated in degree 1, has the correct Hilbert series and has a quotient isomorphic to ${\check{B}(Y,\Lscr\otimes\sigma^*\Lscr(-p),\psi)}$ where $\psi: Y \rightarrow Y$ is some automorphism such that $\psi^3 = \sigma^4$.\\

In \cite{Presotto} it was shown that there is an inclusion 
\begin{equation} \label{eq:NCbirational2} \delta: \check{A} \hookrightarrow \check{A'} \end{equation}
such that the composition $\gamma \circ \delta: \check{A} \hookrightarrow \check{A}^{(2)}$ and $\delta \circ \gamma: \check{A'} \hookrightarrow \check{A}^{\prime (2)}$ induce the identity on the associated function fields. This was done by investigating the following $\mathbb{Z}^2$-algebra $\tilde{A}$:
\[
\tilde{A}_{(i,j),(m,n)}:=
\begin{cases}
\Hom_X(\Oscr_X(-m-2n),\Oscr_X(-i-2j)m_{\tau^{-j}p}\ldots m_{\tau^{-n+1}p}&\text{if $n> j$}\\
\Hom_X(\Oscr_X(-m-2n),\Oscr_X(-i-2j))&\text{if $n \leq j$}
\end{cases}
\]
The construction of the inclusion \eqref{eq:NCbirational2} is then based upon the following observations:
\begin{enumerate}[i)]
\item $\tilde{A}_{(i,0),(m,0)} \cong A_{i,m}$
\item $\tilde{A}_{(0,j),(0,n)} \cong A'_{j,n}$
\item $\tilde{A}_{(i,j),(m,n)} \subset A_{(i+2j,m+2n)}$, hence $\tilde{A}$ contains no non-trivial zero-divisors (because $A$ is a domain).
\item $\dim_k \left( \tilde{A}_{(i,j),(i+1,j-1)} \right)=1$
\end{enumerate}
We let $\delta_{(i,j)}$ be a nonzero element in $\tilde{A}_{(i,j),(i+1,j-1)}$. Then \eqref{eq:NCbirational2} is of the form
\begin{equation} \label{eq:NCbirational4} \tilde{A}_{(i,0),(m,0)} \rightarrow \delta_{(1,i-1)}^{-1} \ldots \delta_{(i,0)}^{-1} \tilde{A}_{(i,0),(m,0)} \delta_{(m,0)} \ldots \delta_{(1,m-1)} \subset \tilde{A}_{(0,i),(0,m)} \end{equation}
The fact that $\gamma \circ \delta$ and $\delta \circ \gamma$ induce the identity on the function fields of $A$ and $A'$ is based upon the following observations:
One can fix nonzero elements $\gamma_{(i,j)}$ in the 1-dimensional vectorspace $\tilde{A}_{(i,j),(i+2,j-1)}$ such that \eqref{eq:NCbirational1} is actually of the form
\begin{equation} \label{eq:NCbirational3} \tilde{A}_{(0,i),(0,m)} \rightarrow \gamma_{(2i-2,1)}^{-1} \ldots \gamma_{(0,i)}^{-1} \tilde{A}_{(0,i),(0,m)} \gamma_{(0,m)} \ldots \gamma_{(2m-2,1)} \subset \tilde{A}_{(2i,0),(2m,0)} \end{equation}
Conversely it was shown that
\begin{equation} 
\label{eq:NCbirational5} \check{A} \cong \bigoplus_{(i,j)} \Hom_{X'}(\Oscr_X'(-2j),\Oscr_{X'}(-2i)\otimes m_{d_i} \ldots m_{d_j})
\end{equation}
where $X' = \QGr(A')$ and
\begin{eqnarray} \label{definitiondi}
d_i = \begin{cases} 
\tau^{-j}p' & \text{if $i=2j$} \\
\tau^{-j}q'& \text{if $i=2j+1$} \\
\end{cases}
\end{eqnarray}
where $p', q' \in Y$ are defined by the following relations in $\Pic(Y)$:
\begin{eqnarray}
\notag p + \tau q' & \sim & [\Lscr] \\
\label{eq:defp'q'} p +  p' & \sim & [\sigma^* \Lscr]
\end{eqnarray}
Moreover the obvious inclusion (induced by \eqref{eq:NCbirational5})
\[ \check{A} \cong \bigoplus_{(i,j)} \Hom_{X'}(\Oscr_X'(-2j),\Oscr_{X'}(-2i)\otimes m_{d_i} \ldots m_{d_j}) \subset \bigoplus_{(i,j)} \Hom_{X'}(\Oscr_X'(-2j),\Oscr_{X'}(-2i)) \cong \check{A'}^{(2)} \]
coincides with \eqref{eq:NCbirational4}. These facts are combined in \cite{Presotto} to conclude that $\gamma \circ \delta$ and $\delta \circ \gamma$ induce the identity on the function fields of $A$ and $A'$.

Finally for further use in this paper we recall from \cite{Presotto} that the $\mathbb{Z}^2$-algebra $\tilde{B}$ was defined as
\[
\tilde{B}_{(i,j),(m,n)}:=
\begin{cases}
\Gamma(Y,\sigma^{*(i+2j)}\Lscr \sigma^{*(i+2j+1)}\Lscr \ldots \sigma^{*(m+2n-1)}\Lscr(-\tau^{-j}p - \ldots - \tau^{-n+1}p ))&\text{if $n> j$}\\
\Gamma(Y,\sigma^{*(i+2j)}\Lscr \sigma^{*(i+2j+1)}\Lscr \ldots \sigma^{*(m+2n-1)}\Lscr)&\text{if $n\leq j$}
\end{cases}
\]
and there is a morphism $\tilde{A} \rightarrow \tilde{B}$ which is an epimorphism in the first quadrant (\cite[Lemma 8.9]{Presotto})

\section{The $\mathbb{Z}^2$-algebras as in \cite{Presotto} are Noetherian.}
\label{sec:Zsquarenoeth}
In this section we show that the $\mathbb{Z}^2$-algebras $\tilde{A}$ as introduced above are Noetherian (in the sense that $\Gr(\tilde{A})$ is locally Noetherian). 
In order to prove that $\tilde{A}$ is Noetherian we work through $\tilde{B}$. By construction there is a surjective $\mathbb{Z}^2$-algebra morfism $\tilde{A}_+ \rightarrow \tilde{B}_+$ (recall \Cref{not:firstquadrant}). We first check that in sufficiently high degrees this map is given by killing a certain collection of ``normalizing'' elements $\{g_{(i,j)} \}$. I.e. we prove the following
\begin{lemma} With the notations as above:
\label{lem:BisAgA}
\begin{enumerate}
\item For all $i,j \in \mathbb{Z}$ the map $\tilde{A}_{(i,j),(i+2,j+1)} \rightarrow \tilde{B}_{(i,j),(i+2,j+1)}$ has a one dimensional kernel. Let $g_{(i,j)}$ be a nonzero element in this kernel.
\item For all $a \geq 2, b \geq 1$ the kernel of $\tilde{A}_{(i,j),(i+a,j+b)} \rightarrow \tilde{B}_{(i,j),(i+a,j+b)}$ is given by $g_{(i,j)}\tilde{A}_{(i+2,j+1),(i+a,j+b)} = \tilde{A}_{(i,j),(i+a-2,j+b-1)}g_{(i+a-2,j+b-1)}$.
\end{enumerate}
\end{lemma}
\begin{proof}
$(1)$ follows from Lemma \ref{lem:dimtilde}. \\
In order to prove $(2)$ we only prove the first equality. The second equality is analogous.
Now for each $i,j \in \mathbb{Z}$ take 
\[ g_{(i,j)} \in \Ker \left( \tilde{A}_{(i,j),(i+2,j+1)} \rightarrow \tilde{B}_{(i,j),(i+2,j+1)}\right) \setminus \{ 0 \} \]
The following diagram shows $ g_{(i,j)}\tilde{A}_{(i+2,j+1),(i+a,j+b)} \subset \Ker \left( \tilde{A}_{(i,j),(i+2,j+1)} \rightarrow \tilde{B}_{(i,j),(i+2,j+1)} \right)$
\begin{center}
\begin{tikzpicture}
\matrix(m)[matrix of math nodes,
row sep=3em, column sep=1em,
text height=1.5ex, text depth=0.25ex]
{ 0 & & \\
 g_{(i,j)}k \otimes \tilde{A}_{(i+2,j+1),(i+a,j+b)} & & g_{(i,j)} \tilde{A}_{(i+2,j+1),(i+a,j+b)} \\
 \tilde{A}_{(i,j),(i+2,j+1)} \otimes \tilde{A}_{(i+2,j+1),(i+a,j+b)} & & \tilde{A}_{(i,j),(i+a,j+b)}\\
 \tilde{B}_{(i,j),(i+2,j+1)} \otimes \tilde{A}_{(i+2,j+1),(i+a,j+b)} & \tilde{B}_{(i,j),(i+2,j+1)} \otimes \tilde{B}_{(i+2,j+1),(i+a,j+b)}& \tilde{B}_{(i,j),(i+a,j+b)} \\
 0 & & 0 \\};
\path[->,font=\scriptsize]
(m-1-1) edge (m-2-1)
(m-2-1) edge (m-3-1)
        edge (m-2-3)
(m-2-3) edge (m-3-3)
(m-3-1) edge (m-3-3)
        edge (m-4-1)
(m-3-3) edge (m-4-3)
(m-4-1) edge (m-4-2)
        edge (m-5-1)
(m-4-2) edge (m-4-3)
(m-4-3) edge (m-5-3);
\end{tikzpicture}
\end{center}
Hence it suffices to show that the alternating sum of the dimensions of the spaces in the right column equals zero. This again follows form Lemma \ref{lem:dimtilde}.
\end{proof}

\begin{corollary}
The collection $\{ g_{(i,j)} \}_{i,j \in \mathbb{Z}}$ is normalizing in the sense that if $a~\in~\tilde{A}_{(i+2,j+1),(m,n)}$ then there exists a unique $a'~\in~\tilde{A}_{(i,j),(m-2,n-1)}$ such that $g_{(i,j)}a=a'g_{(m-2,n-1)}$.
As such, for every right $\tilde{A}$-module $M$ one can consider the right $\tilde{A}$-module $Mg$ defined by $(Mg)_{(i,j)} := M_{(i-2,j-1)}g_{(i-2,j-1)}$.
\end{corollary}

\begin{lemma}
\label{lem:dimtilde}
\begin{eqnarray*}
\dim_k \left( \tilde{A}_{(i,j),(i+a,j+b)} \right) & = & \begin{cases} \frac{a^2 +4ab+ 2b^2 +4a+6b+4}{4} & \textrm{ for $a,b \geq 0$, $a$ even } \\
\frac{a^2 +4ab+ 2b^2 +4a+6b+3}{4} & \textrm{ for $a,b \geq 0$, $a$ odd } \\
 \end{cases}\\
 & \textrm{and} & \\
\dim_k \left( \tilde{B}_{(i,j),(i+a,j+b)} \right) & = & \begin{cases} 2a+3b & \textrm{ for $a,b \geq 0$, $(a,b) \neq (0,0)$ } \\ 1& \textrm{for $(a,b) = (0,0)$} \end{cases} \\ 
\end{eqnarray*} 
\end{lemma}
\begin{proof}
For $\tilde{B}$ the equality is trivial as this corresponds to calculating sections of line bundles on an elliptic curve. For $\tilde{A}$ the computation is done in a similar way as in \cite[\S 6]{PresVdB}.
\end{proof}

In \S \ref{sec:Btildenoeth} we will slightly generalize the arguments in \cite[\S 3]{AV} to show
\begin{theorem} \label{thm:Btildenoeth}
$\tilde{B}_{+}$ is Noetherian.
\end{theorem}
Using this result as well as \Cref{lem:BisAgA} we can prove
\begin{theorem} \label{thm:Atildenoeth}
$\tilde{A}$ is Noetherian
\end{theorem}
\begin{proof}
Obviously the objects $e_{(i,j)}\tilde{A}$ generated $\Gr(\tilde{A})$, hence it suffices to show that these are Noetherian. Let
\begin{equation} \label{eq:inclusionchain} M_0 \subset M_1 \subset \ldots \subset e_{(i,j)}\tilde{A} \end{equation}
be an ascending chain of right $\tilde{A}$ modules. We need to show that this sequence stabilizes. This is done in two steps. First we show the existence of an $N_0 \in \mathbb{N}$ such that $(M_n)_{(i+a,j+b)} = (M_{N_0})_{(i+a,j+b)}$ holds for all $n \geq N_0$, $ab \leq 0$. This is based on Noetherianity of $A$ and $A'$. Next we show the existence of an $N_1 \in \mathbb{N}$ such that $(M_n)_{(i+a,j+b)} = (M_{N_1})_{(i+a,j+b)}$ holds for $a,b \geq 0$. Which is heavily based on \Cref{thm:Btildenoeth}. (This also explains why we only need Noetherianity of $\tilde{B}$ in the first quadrant in \Cref{thm:Btildenoeth}).\\
\underline{Step 1: Convergence of $\left((M_n)_{(i+a,j+b)}\right)_{n \in \mathbb{N}}$ for $ab \leq 0$} \\
As the cases $a \leq 0$ and $b \leq 0$ are analogous, we only prove the case $a \leq 0$. Moreover without loss of generality we can assume $i=j=0$. Recall that $\tilde{A}_{(m,n),(m-1,n+2)}$ is a 1-dimensional vector space and that $\delta_{(m,n)}$ is a nonzero element in this vector space. Moreover as $\tilde{A}$ has no non-trivial zero divisors in the sense that 
\[ \forall a,b,i,j,m,n \in \mathbb{Z}, x \in \tilde{A}_{(a,b),(i,j)}, y \in \tilde{A}_{(i,j),(m,n)}: xy = 0 \Rightarrow x=0 \lor y=0 \]
multiplication by these elements defines embeddings of vectorspaces. By the dimension count in \Cref{lem:dimtilde} the embeddings
\[ \tilde{A}_{(0,0),(a,b)} \hookrightarrow \tilde{A}_{(0,0),(a-1,b+2)}: x \mapsto x \delta_{(a,b)} \]
are in fact isomorphisms. Moreover these isomorphisms induce embeddings \\
${(M_n)_{(a,b)} \hookrightarrow (M_n)_{(a-1,b+2)}}$. In particular if we define $M_n^c \subset e_0 A'$ via
\[ (M_n^c)_m := \{ x \in A'_{0,m} \cong \tilde{A}_{(0,0),(0,m)} \mid x \delta^c  \in (M_n)_{(-c,m+2c)} \} \]
(where we used the shorthand notation $\delta^c = \delta_{(0,m)}\delta_{(-1,m+2)} \ldots \delta_{(-c+1,m+2c-2)}$.)\\
we find an ascending chain of $A'$-submodules of $e_0 A'$. (The $A'$-structure follows from the fact that $A'_{m,m'} = \tilde{A}_{(0,m),(0,m')}$ is isomorphic to $\tilde{A}_{(-c,m+2c),(-c,m'+2c)}$ via $\delta^{-c}(\cdots)\delta^c$). As $A'$ is Noetherian this sequence must stabilize. I.e. for each $n$ there is a natural number $c_n$ such that $M_n^{c_n} = M_n^c$ holds for all $c \geq c_n$. Moreover \eqref{eq:inclusionchain} induces
\[ M_0^{c_0} \subset M_1^{c_1} \subset \ldots \subset e_0 A' \]
This chain must stabilize as well. In particular there is some $N$ such that ${M_n^{c_n}=M_N^{c_N}}$ holds for all $n \geq N$. Going back to \eqref{eq:inclusionchain} and the definition of $M_n^c$ one sees that $c_N \geq c_n$ must hold for all $n \geq N$. Hence $M_n^c = M_N^c$ for all $n \geq N, c \geq c_N$. Similarly the (finitely many!) chains of inclusions
\begin{eqnarray*}
M_0^0 \subset M_1^0 \subset & \ldots &  \subset e_0A' \\
M_0^1 \subset M_1^1 \subset & \ldots &  \subset e_0A' \\
 & \vdots & \\
M_0^{c_N-1} \subset M_1^{c_N-1} \subset & \ldots &  \subset e_0A'
\end{eqnarray*}
must similtanuously stabilize. Hence there exists a $N_0 \in \mathbb{N}$ such that $M_n^c = M_{N_0}^c$ holds for all $n \geq N_0$ and all $c \in \mathbb{N}$. But this is exactly the same as
\[ \forall n \geq N_0, a \leq 0, b \in \mathbb{Z}: (M_n)_{(a,b)} = (M_{N_0})_{(a,b)} \]
\underline{Step 2: Convergence of $\left((M_n)_{(a,b)}\right)_{n \in \mathbb{N}}$ for $a,b \geq 0$} \\
Let $\tilde{A}_{+}$ be defined as in \Cref{not:firstquadrant} and define $M_{n,(\geq 0, \geq 0)}$ via
\begin{equation} \label{eq:geqageqb} \left(M_{(\geq a, \geq b)}\right) = \begin{cases}
M_{(i,j)} & \textrm{if $i \geq a, j \geq b$} \\
0 & \textrm{else}
\end{cases}\end{equation}

then $\left((M_{n,(\geq 0, \geq 0)})\right)_{n \in \mathbb{N}}$ defines an ascending chain of $\tilde{A}_{+}$-submodules of $e_{(0,0)} \tilde{A}_{+}$. Hence it suffices to show that every $\tilde{A}_{+}$-submodule of $e_{(0,0)} \tilde{A}_{+}$ is finitely generated in the sense that it can be written as a quotient of some finite direct sum of projective generators $\displaystyle \bigoplus_{n=1}^{n_0} e_{(i_n, j_n)} \tilde{A}_{+}$. 
This is shown in a similar way as in \cite[Lemma 8.2]{ATV1}:\\

By way of contradiction suppose that there is some submodule $L \subset e_{(0,0)} \tilde{A}_{+}$ which is not finitely generated. Using Zorn's Lemma $L$ can be chosen maximal with respect to the inclusion of submodules. The quotient $\overline{A} = e_{(0,0)} \tilde{A}_{+}/L$ as well as all its submodules must hence be finitely generated. Now consider the following diagram

\begin{center}
\begin{tikzpicture}
\matrix(m)[matrix of math nodes,
row sep=3em, column sep=3em,
text height=1.5ex, text depth=0.25ex]
{ 0 & L & e_{(0,0)} \tilde{A}_{+} & \overline{A} & 0 \\
 0 & L & e_{(0,0)} \tilde{A}_{+} & \overline{A} & 0\\};
\path[->,font=\scriptsize]
(m-1-1) edge (m-1-2)
(m-1-2) edge (m-1-3)
        edge node[left]{$\cdot g$} (m-2-2)
(m-2-2) edge (m-2-3)
(m-2-1) edge (m-2-2)
(m-1-3) edge node[left]{$\cdot g$} (m-2-3)
        edge (m-1-4)
(m-1-4) edge (m-1-5)
        edge node[left]{$\cdot g$} (m-2-4)
(m-2-4) edge (m-2-5)
(m-2-3) edge (m-2-4);
\end{tikzpicture}
\end{center}
Applying the Snake Lemma (and the fact that the $g_{(i,j)}$ are normalizing, non-zero divisors) provides an exact sequence
\[ 0 \rightarrow K \rightarrow L/Lg \overset{\epsilon}{\rightarrow} e_{(0,0)} \tilde{A}_{+}/ g_{(0,0)} \tilde{A}_{+}  \]
where $K = \Ker(\overline{A} \overset{\cdot g}{\rightarrow} \overline{A} )$. As mentioned above, all submodules of $\overline{A}$ are finitely generated, hence $K$ is finitely generated. We also claim $\im(\epsilon) \subset e_{(0,0)} \tilde{A}_{+}/ g_{(0,0)} \tilde{A}_{+}$ is finitely generated. From this claim it follows that $L/Lg$ is finitely generated as an extension of finitely generated modules and hence $L$ is finitely generated as well (as the elements $g_{(i,j)}$ live in strictly positive degrees). Thus it suffices to prove the claim. First we prove the existence of a map
\[ \bigoplus_{n=1}^{n_1} e_{(i_n, j_n)} \tilde{A}_{+} \rightarrow \im(\epsilon) \]
which is surjective in degrees $(a,b)$ with $a \geq 2$, $b \geq 1$. Note that the elements in $\im(\epsilon)$ living in these degrees actually form a $\tilde{B}_{+}$-submodule by \Cref{lem:BisAgA}. By \Cref{thm:Btildenoeth} it must be finitely generated as a $\tilde{B}_{+}$-module and hence also as an $\tilde{A}_{+}$-module. Next we prove the existence of a map
\begin{equation} \label{eq:map0} \bigoplus_{n=1}^{n_2} e_{(i_n, j_n)} \tilde{A}_{+} \rightarrow \im(\epsilon) \end{equation}
which is surjective in degrees $(a,b)$ with $a < 2$ or $b = 0$. Note that $(\im(\epsilon)_{(a,0)})_{a \in \mathbb{N}}$ forms an $A$-submodule of $e_0A$. Hence there is an epimorphism
\[ \bigoplus_{m=1}^{m_2} e_{(i_m)} A \rightarrow (\im(\epsilon)_{(a,0)})_{a \in \mathbb{N}} \]
giving rise to a map
\begin{equation} \label{eq:map1} \bigoplus_{m=1}^{m_2} e_{(i_m,0)} \tilde{A}_{+} \rightarrow \im(\epsilon) \end{equation}
which is surjective in degrees $(a,0)$. Similarly we can construct an $m_2' \in \mathbb{N}$ and a map 
\begin{equation} \label{eq:map2} \bigoplus_{m=1}^{m_2'} e_{(0,j_m)} \tilde{A}_{+} \rightarrow \im(\epsilon) \end{equation}
which is surjective in degrees $(0,b)$. Finally using the fact that $(\tilde{A}_{+})_{(1,b), (1,b')}$ is isomorphic to $(\tilde{A}_{+})_{(0,b+2), (0,b'+2)} \cong A'_{b+2,b'+2}$ via $\delta^{-1}(\cdots)\delta$, we get a structure for $(\im(\epsilon)_{(1,b)})_{b \in \mathbb{N}}$ as a $A'$ submodule of $e_0 A'$. Hence there is a surjection
\[ \bigoplus_{m=1}^{m''_2} e_{(j_m)} A' \rightarrow (\im(\epsilon)_{(1,b)})_{b \in \mathbb{N}} \]
Note that we can assume $j_m \geq 2$ for all $m$ such that we find a map
\begin{equation} \label{eq:map3} \bigoplus_{m=1}^{m''_2} e_{(1,j_m-2)} \tilde{A}_{+} 
\rightarrow \im(\epsilon) \end{equation}
which is surjective in degrees $(1,b)$. Finally combining \eqref{eq:map1}, \eqref{eq:map2} and \eqref{eq:map3} we find the required map \eqref{eq:map0}, proving the claim.

\end{proof}

\section{Ample $\mathbb{Z}^2$-sequences and Noetherianity of $\tilde{B}_{+}$}
\label{sec:Btildenoeth}
This goal of this section is to prove a $\mathbb{Z}^2$-version of \cite[Theorem 1.4]{AV}. For this we introduce some ad hoc definitions:
\begin{definition}
Let $Y$ be a Noetherian variety. A projective $\mathbb{Z}^2$-sequence on $Y$ consists of two collections of line bundles $\{ \Lscr_{(i,j)} \}_{i,j \in \mathbb{Z}}, \{ \Gscr_{(i,j)} \}_{i,j \in \mathbb{Z}}$ satisfying
\begin{equation} \label{eq:LscrGscr}
\Lscr_{(i,j)} \Gscr_{(i+1,j)} = \Gscr_{(i,j)} \Lscr_{(i,j+1)}
\end{equation}
\end{definition}
\begin{remark}
We refer to these $\mathbb{Z}^2$-sequences as ``projective'' because they obviously generalize projective sequences as in \cite{Polishchuk}
\end{remark}
Of particular importance are ample $\mathbb{Z}^2$-sequences:

Given a projective $\mathbb{Z}^2$-sequence $\{ \Lscr_{(i,j)} \}_{i,j \in \mathbb{Z}}, \{ \Gscr_{(i,j)} \}_{i,j \in \mathbb{Z}}$, we associate a sheaf of $\mathbb{Z}^2$-algebras $\Bscr$ to it via
\begin{equation} \label{eq:defBscr}
\Bscr_{(i,j),(m,n)} = \begin{cases} \Lscr_{(i,j)} \Lscr_{(i+1,j)} \ldots \Lscr_{(m-1,j)}\Gscr_{(m,j)} \Gscr_{(m,j+1)} \ldots \Gscr_{(m,n-1)} & \textrm{ if $i \leq m, j \leq n$} \\
\left( \Lscr_{(m,j)} \Lscr_{(m+1,j)} \ldots \Lscr_{(i-1,j)} \right)^{-1} \Gscr_{(m,j)} \Gscr_{(m,j+1)} \ldots \Gscr_{(m,n-1)} & \textrm{ if $i > m, j \leq n$} \\
\Lscr_{(i,j)} \Lscr_{(i+1,j)} \ldots \Lscr_{(m-1,j)} \left( \Gscr_{(m,n)} \Gscr_{(m,n+1)} \ldots \Gscr_{(m,j-1)}\right)^{-1} & \textrm{ if $i \leq m, j > n$} \\
\left( \Lscr_{(m,j)} \Lscr_{(m+1,j)} \ldots \Lscr_{(i-1,j)} \right)^{-1}\left( \Gscr_{(m,n)} \Gscr_{(m,n+1)} \ldots \Gscr_{(m,j-1)}\right)^{-1} & \textrm{ if $i > m, j > n$} \\
\end{cases}
\end{equation}
We also associate a $\mathbb{Z}^2$-algebra $B = B \left( \{ \Lscr_{(i,j)} \}_{i,j \in \mathbb{Z}}, \{ \Gscr_{(i,j)} \}_{i,j \in \mathbb{Z}} \right)$ to the sequence via
\begin{equation} \label{eq:defBijmn} B_{(i,j),(m,n)} = \begin{cases} \Gamma(Y,\Bscr_{(i,j),(m,n)}) & \textrm{if $i \leq m, j \leq n$} \\ 0 & \textrm{else}  \end{cases} \end{equation}
\begin{definition} \label{def:amplesequence}
Let $\{ \Lscr_{(i,j)} \}_{i,j \in \mathbb{Z}}, \{ \Gscr_{(i,j)} \}_{i,j \in \mathbb{Z}}$ be a projective $\mathbb{Z}^2$-sequence and let $\Bscr$ be the associated sheaf of $\mathbb{Z}^2$-algebras, then we say the sequence is an ample sequence if for each coherent sheaf $\Fscr$ on $Y$ and for each $i,j \in \mathbb{Z}$ there exist $i_0,j_0 \in \mathbb{Z}$ such that 
\[ H^q(Y, \Fscr \otimes_Y \Bscr_{(i,j),(m,n)}) = 0 \]
holds for all $q > 0$ and $m \geq i_0, n \geq j$ or $m \geq i, n \geq j_0$.
\end{definition}
The following is immediate
\begin{proposition} \label{prp:BtildeisBscr}
With the notation as in \S \ref{sec:summarycon}: define $\Lscr_{(i,j)} = \sigma^{*(i+2j)}\Lscr$ and $\Gscr_{(i,j)} = \sigma^{*(i+2j)}\Lscr\sigma^{*(i+2j+1)}\Lscr(-\tau^{-j}p)$. Then $\{ \Lscr_{(i,j)} \}_{i,j \in \mathbb{Z}}, \{ \Gscr_{(i,j)} \}_{i,j \in \mathbb{Z}}$ is an ample $\mathbb{Z}^2$-sequence and $B\left( \{ \Lscr_{(i,j)} \}_{i,j \in \mathbb{Z}}, \{ \Gscr_{(i,j)} \}_{i,j \in \mathbb{Z}} \right) = \tilde{B}_{+}$.
\end{proposition}
\begin{proof}
It is obvious that $\tilde{B}_{+}$ equals the $\mathbb{Z}^2$-algebra associated to the sequence and that \eqref{eq:LscrGscr} is satisfied. It only remains to show that the associated sheaf of $\mathbb{Z}^2$-algebras $\Bscr$ satisfies the condition in \Cref{def:amplesequence}. For this let $\Fscr$ be a coherent sheaf on $Y$. As $Y$ is a nonsingular elliptic curve, \cite[Exercice II.6.11(c)]{hartshorne} gives us the existence of a line bundle $\Lscr$ and a torsion sheaf $\Tscr$ together with a short exact sequence:
\[ 0 \rightarrow \Lscr^{\oplus r} \rightarrow \Fscr \rightarrow \Tscr \rightarrow 0 \]
where $r$ is the rank of $\Fscr$. As $\Tscr$ is torsion, its support is zero dimensional and hence $\Tscr \otimes \Bscr_{(i,j),(m,n)}$ has no higher cohomology. In particular 
\[ H^q(Y,\Fscr \otimes \Bscr_{(i,j),(m,n)}) \cong H^q(Y,\Lscr \otimes \Bscr_{(i,j),(m,n)})^{\oplus r} \]
Again using the fact that $Y$ is one dimensional, we only need to show $H^1(Y,\Lscr \otimes \Bscr_{(i,j),(m,n)})=0$ for sufficiently high $m,n$. By Riemann-Roch the latter follows when $\deg(\Lscr \otimes \Bscr_{(i,j),(m,n)}) \geq 2g-2$ where $g$ is the genus of $Y$. Finally this condition is satisfied for $m,n$ big enough as $\deg(\Bscr_{(i,j),(m,n)})$ is a strictly increasing in function of the variables $m$ and $n$.
\end{proof}

In particular \Cref{thm:Btildenoeth} follows from \Cref{prp:BtildeisBscr} and the following:

\begin{theorem} \label{thm:Bscrnoeth}
Let $\{ \Lscr_{(i,j)} \}_{i,j \in \mathbb{Z}}, \{ \Gscr_{(i,j)} \}_{i,j \in \mathbb{Z}}$ be an ample $\mathbb{Z}^2$ sequence on a Noetherian, projective scheme $Y$, then $B := B\left( \{ \Lscr_{(i,j)} \}_{i,j \in \mathbb{Z}}, \{ \Gscr_{(i,j)} \}_{i,j \in \mathbb{Z}} \right)$ is Noetherian in the sense that $\Gr(B)$ is a locally Noetherian category.
\end{theorem}

\begin{remark}
It was shown in \cite{chan} that, under suitable conditions, twisted bi-homogeneous coordinate rings are Noetherian. The $(\check{-})$-construction which turns graded algebras into $\mathbb{Z}$-algebras can be generalized to turn bi-homogeneous algebras into $\mathbb{Z}^2$-algebras, tri-homogeneous algebras into $\mathbb{Z}^3$-algebras, etc. Moreover it is not hard to see that in this way we can turn a twisted bi-homogeneous coordinate ring into a $\mathbb{Z}^2$-algebra of the form \eqref{eq:defBijmn}. As such \Cref{thm:Bscrnoeth} is a generalization of \cite[Theorem 5.2]{chan}.
\end{remark}

The proof of \Cref{thm:Bscrnoeth} follows from a chain of lemmas:

\begin{lemma} \label{lem:Bgenbysections}
Let $\Fscr$ be a coherent sheaf on $Y$, then for all $i,j \in \mathbb{Z}: \Fscr \otimes \Bscr_{(i,j),(m,n)}$ is generated by its global sections for sufficiently large $m,n$, in the sense that there exist $i_0, j_0 \in \mathbb{Z}$ such that $\Fscr \otimes_Y \Bscr_{(i,j),(m,n)}$ is generated by global sections whenever $m \geq i, n \geq j_0$ or $m \geq i_0, n \geq j$
\end{lemma}
\begin{proof}
This is a straightforward generalization of \cite[Proposition 3.2.ii and Lemma 3.3]{AV}.
\end{proof}
\begin{lemma}\label{lem:twopairfunctors}
We have two pairs of functors between $\Gr(B)$, $\Gr(\Bscr)$ and $\Qcoh(Y)$:
\[
\begin{tikzpicture}
\node (b)at (-4,0) {$\Gr(B)$};
\node (c) at (0,0) {$\Gr(\Bscr)$};
\node (d) at (-2,1.2) {$\widetilde{(-)} := - \otimes_B \Bscr$};
\node (e) at (-2,-1.1) {$\Gamma_* := \bigoplus_{n \in \mathbb{Z}} \Gamma(Y,(-)_n)$};
\node (f) at (4,0) {$\Qcoh(Y)$};
\node (g) at (2,1.2) {$(-)_{(0,0)}$};
\node (h) at (2,-1.1) {$- \otimes_Y e_{(0,0)} \Bscr$};
\draw[->] (b) to[bend left](c) ;
\draw[->] (c) to[bend left] (b);
\draw[->] (c) to[bend left](f) ;
\draw[->] (f) to[bend left] (c);
\end{tikzpicture} \]
Moreover these satisfy the following properties:
\begin{enumerate}
\item $(-)_{(0,0)}$ and $- \otimes_Y e_{(0,0)} \Bscr$ are quasi-inverses and define an equivalence of categories
\item $\widetilde{(-)}$ is left adjoint to $\Gamma_*$
\item $\Gamma_*$ is exact modulo torsion modules
\item Let $M \in \Gr(B)$ and define $\overline{M} := \Gamma_*(\widetilde{M})$. Then there is a natural map $M \rightarrow \overline{M}$, moreover the kernel and cokernel of this map are torsion.
\item $\widetilde{(-)}$ is exact
\end{enumerate}
\end{lemma}
\begin{proof}
\begin{enumerate}
\item This is standard and follows from the fact that $\Bscr$ is strongly graded (i.e. $\Bscr_{(a,b),(i,j)}\Bscr_{(i,j),(m,n)} = \Bscr_{(a,b),(m,n)}$ holds for all $a,b,i,j,m,n \in \mathbb{Z}$).
\item The functor $\widetilde{(-)}$ is defined by considering $B$ as well as each graded $B$-module as constant sheaves on $Y$. As such $\Bscr$ obtains a natural left $B$-structure. The fact that $\widetilde{(-)}$ is left adjoint to $\Gamma_*$ follows immediately from this construction.
\item Analogous to \cite[Lemma 3.7.(ii)]{AV}
\item The existence of the natural map $M \rightarrow \overline{M}$ follows from the adjunction in (2). The fact that the kernel and cokernel are torion are a straightforward generalization of  \cite[Lemma 3.13.(i) and (iii)]{AV}
\item Analogous to \cite[Lemma 3.13.(iv)]{AV} (also using \cite[Lemma 3.7.(iii)]{AV})
\end{enumerate}
\end{proof}
\begin{remark} \label{rem:def:coherent}
The above lemma implies that up to isomorphism any graded $\Bscr$-module $\Mscr$ can be written as $\Fscr \otimes_Y e_{(0,0)} \Bscr$ for some quasi-coherent sheaf $\Fscr$. In the special case that $\Fscr$ is coherent, $\Mscr$ is called coherent as well.
\end{remark}

\begin{corollary} \label{cor:coherentshifts}
Let $\Mscr$ be a coherent $\Bscr$-module. Then there exist $i,j,n \in \mathbb{Z}$, $n \geq 0$ such that $\Mscr$ is a quotient of $\left( e_{(i,j)}\Bscr \right)^{\oplus n}$
\end{corollary}
\begin{proof}
There is a coherent sheaf $\Fscr$ such that $\Mscr = \Fscr \otimes_Y e_{(0,0)} \Bscr$. By \Cref{lem:Bgenbysections} there are $i,j \in \mathbb{Z}$ such that $\Fscr \otimes \Bscr_{(0,0),(i,j)}$ is generated by global sections. As such $\Mscr$ is a quotient of $\Gamma(Y,\Fscr \otimes \Bscr_{(0,0),(i,j)}) \otimes_k e_{(i,j)} \Bscr$
\end{proof}

\begin{lemma} \label{lem:radfingen}
For each $i,j \in \mathbb{Z}, \forall a,b \in \mathbb{Z} \cup \{ - \infty \}$ we have that $\left(e_{(i,j)}B\right)_{(\geq a, \geq b)}$ is finitely generated. (Recall: the notation $M_{\geq a, \geq b}$ was introduced in \eqref{eq:geqageqb})
\end{lemma}
\begin{proof} (inspired by \cite[p.261-262]{AV})\\
Without loss of generality we assume $i=j=0$. Moreover replacing $a,b$ by $\max(a,0)$, $\max(b,0)$ we can assume $a,b \geq 0$.
By \Cref{lem:Bgenbysections} there is an $m_0 \geq a$ such that $\Bscr_{(0,0),(m_0,b)}$ is generated by sections. I.e. there is a short exact sequence
\[ 0 \rightarrow \Fscr \rightarrow B_{(0,0),(m_0,b)} \otimes_k \Oscr_Y \rightarrow \Bscr_{(0,0),(m_0,b)} \rightarrow 0\]
As $\{ \Lscr_{(i,j)} \}_{i,j \in \mathbb{Z}}, \{ \Gscr_{(i,j)} \}_{i,j \in \mathbb{Z}}$ is an ample sequence (see \Cref{def:amplesequence}) there is $m_1 > m_0$ such that $H^1(Y,\Fscr \otimes_Y \Bscr_{(m_0,b),(m,n)})=0$ for all $m \geq m_1$ and $n \geq b$. In particular, applying $\Gamma_*$ to the above sequence provides a surjective morphism for all $m \geq m_1$ and $n \geq b$
\[ B_{(0,0),(m_0,b)} \otimes_k B_{(m_0,b),(m,n)} \twoheadrightarrow B_{(0,0),(m,n)} \]
I.e. the natural map
\[ B_{(0,0),(m_0,b)} \otimes_k e_{(m_0,b)} B \rightarrow \left(e_{(0,0)}B\right)_{(\geq a, \geq b)} \]
is surjective in degrees $(m,n)$ with $m \geq m_1$ and $n \geq b$. Similarly there is a $n_0\geq b$ and $n_1 > n_0$ such that
\[ B_{(0,0),(a,n_0)} \otimes_k e_{(a,n_0)} B \rightarrow \left(e_{(0,0)}B\right)_{(\geq a, \geq b)} \]
is surjective in degrees $(m,n)$ with $m \geq a$ and $n \geq n_0$. Combining the above two morphisms we find a surjective morphism
\[ \bigoplus_{\substack{a \leq m < m_1 \\ b \leq n < n_1 }} B_{(0,0),(m,n)} \otimes_k e_{(m,n)}B \twoheadrightarrow \left(e_{(0,0)}B\right)_{(\geq a, \geq b)} \]
As $\dim_k(B_{(0,0),(m,n)}) < \infty$ for all $m,n$ this finishes the proof.
\end{proof}
\begin{lemma} \label{lem:fingenSES}
Let $0 \rightarrow M \rightarrow M' \rightarrow M'' \rightarrow 0$ be an exact sequence in $\Gr(B)$.
\begin{enumerate}
\item If $M$ and $M''$ are finitely generated, then so is $M'$.
\item If $M'$ is finitely generated and there exist $a,b,c,d \in \mathbb{Z}$ such that $M''_{(i,j)}=0$ if $i \geq c$ and $j \geq d$ or if $i < a$ or if $j<b$ (i.e. there exists a $N \in \Gr(B)$ such that $M'' = N_{(\geq a, \geq b)}/N_{(\geq c, \geq d)}$), then $M$ is finitely generated.
\end{enumerate}
\end{lemma}
\begin{proof}
\begin{enumerate}
\item This is standard.
\item By induction it suffices to show the lemma in case $c=a, d=b+1$ or $c=a+1, d=b$. Without los of generality we may assume the latter, i.e. $M''_{(i,j)}=0$ if $i\neq a$ or $j < b$. As $M'$ is finitely generated there exists as surjective morphism
\begin{equation} \label{eq:M'fingen} \bigoplus_{n=1}^N e_{(i_n,j_n)} B \twoheadrightarrow M' \end{equation}
For each $n$, consider the composition
\[ e_{(i_n,j_n)} B \rightarrow M'' \]
As this is a morphism of $B$-modules, it can only be nonzero if $i_n =a, j_n \geq b$. Moreover for each such $n$, $\left( e_{(a,j_n)} B \right)_{(\geq a+1, \geq j_n)}$ is finitely generated by \Cref{lem:radfingen}, say
\begin{equation}
\bigoplus_{x_n=1}^{x_{n,0}} e_{(i_{x_n},j_{x_n})}B \twoheadrightarrow 
\left( e_{(a,j_n)} B \right)_{(\geq a+1, \geq j_n)}
\end{equation}
Combining the above we obtain a morphism
\begin{equation} \label{eq:surjectivenotonline}
\left( \bigoplus_{\substack{n=1 \\ i_n \neq a \textrm{ or } j_n < b}}^N e_{(i_n,j_n)} B \right) \oplus \left( \bigoplus_{\substack{n=1 \\ i_n =a, j_n \geq b}}^N \left( \bigoplus_{x_n=1}^{x_{n,0}} e_{(i_{x_n},j_{x_n})}B \right) \right) \rightarrow M 
\end{equation}
and by construction this morphism is surjective in all degrees $(i,j)$ for which $i \neq a$ or $j < b$.

Now consider the $\mathbb{Z}$-algebra $B'$ defined by 
\[ B'_{n,m} := B_{(a,b+n),(a,b+m)} \]
By construction $B'_{n,m}$ is the twisted homogeneous coordinate ring with respect to the ample sequence $\left( \Gscr_{(a,b+n)} \right) _{n \in \mathbb{Z}}$. It is standard that $B'$ is a Noetherian $\mathbb{Z}$-algebra (see for example \cite{AV} or \cite{Polishchuk}). Moreover define $B'$-modules $L$ and $L'$ by
\[ L_m := M_{(a,b+m)} \textrm{ and } L'_m = M'_{(a,b+m)} \]
Recall that $M'$ was finitely generated as in \eqref{eq:M'fingen}. For each $n$ we know that $\left( e_{(i_n,j_n)} B \right)_{(\geq a, \geq b)}$ is finitely generated, say
\[ \bigoplus_{u_n=1}^{u_{n,0}} e_{(i_{u_n},j_{u_n})}B \twoheadrightarrow \left( e_{(i_n,j_n)} B \right)_{(\geq a, \geq b)}\]
As such we obtain that $L'$ is finitely generated. For the interested reader we mention that the explicit surjective morphism is given by
\[ \bigoplus_{n=1}^N \left( \bigoplus_{\substack{u_n=1 \\ i_{u_n} =a, j_{u_n} \geq b}}^{u_{n,0}} e_{j_{u_n}-b}B'\right) \twoheadrightarrow L' \]
As $B'$ is Noetherian, $L \subset L'$ is finitely generated as well, say 
\[ \bigoplus_{v=1}^{v_0} e_{j_v} B' \twoheadrightarrow L \]
The induced morphism
\begin{equation}
\label{eq:surjectiveonline}
\bigoplus_{v=1}^{v_0} e_{(a,b+j_v)} B \twoheadrightarrow M
\end{equation}
is surjective in all degrees $(i,j)$ with $i =a, j \geq b$. Combining \eqref{eq:surjectivenotonline} and \eqref{eq:surjectiveonline} we obtain that $M$ is finitely generated as required.
\end{enumerate}
\end{proof}
The following lemma will be crucial in the proof of \Cref{thm:Bscrnoeth}:
\begin{lemma} \label{lem:cohfingen}
Let $M \in \Gr(B)$ be such that $\widetilde{M}$ is coherent, then for each $a,b \in \mathbb{Z}$: $\overline{M}_{(\geq a, \geq b)}$ is a finitely generated graded $B$-module.
\end{lemma}
\begin{proof} (inspired by \cite[Lemma 3.17]{AV})
As $\tilde{M}$ is coherent, \Cref{cor:coherentshifts} provides us with $i,j,n \in \mathbb{Z}$, $n \geq 0$ and a surjective morphism
\[ \left( e_{(i,j)}B \right)^{\oplus n} \twoheadrightarrow \tilde{M} \]
Let $\Kscr$ be the kernel of this morphism. Then we have a long exact sequence
\[ 0 \rightarrow \Gamma_*(\Kscr) \overset{f}{\rightarrow} e_{(i,j)}B^{\oplus n} \rightarrow \overline{M} \rightarrow H^1(\Kscr) \]
Truncation turns this into an exact sequence
\[ 0 \rightarrow \coker(f)_{(\geq a, \geq b)} \rightarrow \overline{M}_{(\geq a, \geq b)} \rightarrow H^1(\Kscr)_{(\geq a, \geq b)} \]
$\coker(f)_{(\geq a, \geq b)}$ is finitely generated as a quotient of $e_{(i,j)}B^{\oplus n}_{(\geq a, \geq b)}$, which in turn is finitely generated by \Cref{lem:radfingen}. $\Kscr$ is coherent as a $\Bscr$-submodule of $e_{(i,j)}\Bscr^{\oplus n}$, hence by the definition of an ample sequence, $H^1(\Kscr)_{(\geq a, \geq b)}$ is concentrated in finitely many degrees. As $Y$ is projective, $H^1(\Kscr)_{(\geq a, \geq b)}$ is finite dimensional and in particular finitely generated. The result now follows from \Cref{lem:fingenSES}.
\end{proof}

We can now finish the proof of the main theorem of this section
\begin{proof}[Proof of \Cref{thm:Bscrnoeth}]
As $\{ e_{(i,j)}B \mid i,j \in \mathbb{Z} \}$ obviously serves as a set of generators for $\Gr(B)$ it suffices to show that these modules are Noetherian. Hence let $M$ be a submodule of some $e_{(a,b)}B$. By \Cref{lem:twopairfunctors}(5) this induces an embedding
\[ \tilde{M} \hookrightarrow e_{(a,b)} \Bscr \]
As such $\tilde{M}$ is coherent and \Cref{lem:cohfingen} implies $\overline{M}_{(\geq a, \geq b)}$ is a finitely generated module. As $M = M_{(\geq a, \geq b)}$ the natural map $M \rightarrow \overline{M}$ factors through $\overline{M}_{(\geq a, \geq b)}$. Now consider the diagram
\begin{center}
\begin{tikzpicture}
\matrix(m)[matrix of math nodes,
row sep=3em, column sep=3em,
text height=1.5ex, text depth=0.25ex]
{ M & e_{(a,b)}B \\
 \overline{M}_{(\geq a, \geq b)} & \overline{e_{(a,b)}B} \\};
\path[->,font=\scriptsize]
(m-1-1) edge (m-1-2)
        edge (m-2-1)
(m-1-2) edge (m-2-2)
(m-2-1) edge (m-2-2);
\end{tikzpicture}
\end{center}
As the upper horizontal arrow and the right vertical arrow are injective, so is the left vertical arrow. We obtain a short exact sequence
\begin{equation} \label{eq:fingenSES} 0 \rightarrow M \rightarrow \overline{M}_{(\geq a, \geq b)} \rightarrow \left(\overline{M}/M \right)_{(\geq a, \geq b)} \rightarrow 0 \end{equation}
By \Cref{lem:twopairfunctors}(4) we have that $\overline{M}/M$ and hence also $\left(\overline{M}/M \right)_{(\geq a, \geq b)}$ is torsion. Being a quotient of a finitely generated module, $\left(\overline{M}/M \right)_{(\geq a, \geq b)}$ is also finitely generated and hence concentrated in finitely many degrees. In particular there exist $c,d \in \mathbb{Z}$ such that $\left(\left(\overline{M}/M \right)_{(\geq a, \geq b)}\right)_{(i,j)}=0$ for $i > c, j > d$. As such \eqref{eq:fingenSES} satisfies the assumptions in \Cref{lem:fingenSES}, implying $M$ is finitely generated.
\end{proof}

\section{Proof of \Cref{thm:commonblowup}}
\label{sec:asblowups}
In this section we give the proof of \Cref{thm:commonblowup}. For this we fix the notations $A, A', \gamma, \delta, Y, \Lscr, \sigma, p, p', q', \tilde{A}$ and $\tilde{B}$ as in \S \ref{sec:summarycon}. As a result of \Cref{thm:Atildenoeth} we can apply \Cref{thm:equivalencediagonal} to the $\mathbb{Z}^2$-algebra $\tilde{A}$. Hence for each diagonal-like $\Delta: \mathbb{Z} \rightarrow \mathbb{Z}^2$ there is an equivalence of categories $\QGr(\tilde{A}) \cong \QGr(\tilde{A}_\Delta)$. Throughout this section we focus on a specific choice of $\Delta$: $\Delta(i)=(2i,i)$. For this choice of $\Delta$ we have the following:
\[ \left( \tilde{A}_\Delta \right)_{i,j} = \tilde{A}_{(2i,i),(2j,j)} = \begin{cases} \Hom_X(\Oscr_X(-4j),\Oscr_X(-4i)m_{\tau^{-i}p}\ldots m_{\tau^{-j+1}p} & \textrm{ if $j \geq i$} \\ 0 & \textrm{ else } \end{cases} \]

In particular there is an obvious inclusion 
\begin{equation} \label{eq:TinA4}
\tilde{A}_\Delta \hookrightarrow \widecheck{(A^{(4)})}
\end{equation} 
Moreover $\widecheck{(A^{(4)})}$ is 1-periodic as it is induced by the graded algebra $A^{(4)}$. The equality
\[ o_X(4)m_d = m_{\tau d}o_X(4) \]
implies that $\tilde{A}_\Delta$ is compatible with this 1-periodicity. I.e. there is a graded algebra $T$ such that $\tilde{A}_\Delta = \check{T}$ and \eqref{eq:TinA4} is induced by an inclusion
\[ T \hookrightarrow A^{(4)} \]
By construction $T$ is the subalgebra of $A^{(4)}$ generated by the elements in $A^{(4)}_1= A_4$ whose images in $A^{(4)}/g$ lie in $\Gamma(Y,\Lscr \sigma^*\Lscr \sigma^{*2}\Lscr \sigma^{*3}\Lscr (-p))$. This is exactly the definition of the noncommutative blow-up $A^{(4)}(p)$ as in \cite{RSS2}.\\ \\

Next we need to show that $\tilde{A}_\Delta = \check{T}$ can also be seen as a noncommutative blow-up of $A'$. I.e. we need to show the existence of a divisor $d'$ such that $T=A^{\prime (3)}(d')$. We claim that $d'=p'+q'$ with $p'$ and $q'$ as in \eqref{eq:defp'q'} will do the job. First note that 
\begin{equation} \label{eq:TinA'3} \delta^{-2i} (-) \delta^{2j}: (\tilde{A}_\Delta)_{i,j} = \tilde{A}_{(2i,i),(2j,j)} \rightarrow \tilde{A}_{(0,3i),(0,3j)} \cong A'_{3i,3j} \end{equation}
defines an inclusion $\check{T} = \tilde{A}_\Delta \hookrightarrow \widecheck{A^{\prime (3)}}$. We need to show that (up to replacing the $\delta_{(i,j)}$ by some scalar multiples) \eqref{eq:TinA'3} is compatible with the 1-periodicity of $\widecheck{A^{\prime (3)}}$ and $\check{T}$ such that there is an induced inclusion $T \hookrightarrow A^{\prime (3)}$. To this we need to show that the periodicity isomorphisms $\tilde{A}_{(0,3i),(0,3j)} \cong \tilde{A}_{(0,3i+3),(0,3j+3)}$ and $\tilde{A}_{(2i,i),(2j,j)} \cong \tilde{A}_{(2i+2,i+1),(2j+2,j+1)}$, induced by $A^{\prime (3)}$ and $T$ respectively, are compatible in the sense that the following diagram commutes

\begin{equation} \label{eq:periodicitycommute}
\begin{tikzpicture}
\matrix(m)[matrix of math nodes,
row sep=3em, column sep=6em,
text height=1.5ex, text depth=0.25ex]
{ \tilde{A}_{(2i,i),(2j,j)} & \tilde{A}_{(0,3i),(0,3j)} \\
\tilde{A}_{(2i+2,i+1),(2j+2,j+1)} & \tilde{A}_{(0,3i+3),(0,3j+3)} \\};
\path[->,font=\scriptsize]
(m-1-1) edge node[above]{$\delta^{-2i} (-) \delta^{2j}$}(m-1-2)
        edge node[left]{$\cong$}(m-2-1)
(m-1-2) edge node[right]{$\cong$}(m-2-2)
(m-2-1) edge node[below]{$\delta^{-2i-2} (-) \delta^{2j+2}$}(m-2-2);
\end{tikzpicture}
\end{equation}
As all morphisms in the above diagram are compatible with the algebra structure of $\tilde{A}$ and as $T$ and $A'$ are generated in degree 1, it suffices to check that the above diagram commutes for $j=i+1$.

Recall that there is a central element $g' \in A'_3$ such that $A'/g'A' \cong B'$ where $B'$ is the twisted homogeneous coordinate ring of $(Y, \Lscr \otimes \sigma^* \Lscr (-p), \psi)$. Let $g'_j$ be the corresponding element in $\check{A'}_{j,j+3}$. Then for each $i$ there is an exact sequence
\begin{equation} \label{eq:decomA'}
0 \rightarrow k g'_{3i} \rightarrow \tilde{A}_{(0,3i), (0,3i+3)} \rightarrow \tilde{B}_{(0,3i), (0,3i+3)} \rightarrow 0
\end{equation}
Similarly by \Cref{lem:BisAgA} for each $i$ there is an element $g_{(2i,i)} \in \tilde{A}_{(2i,i),(2i+2,i+1)}$ together with an exact sequence
\begin{equation} \label{eq:decomADelta}
0 \rightarrow k g_{(2i,i)} \rightarrow \tilde{A}_{(2i,i), (2i+2,i+1)} \rightarrow \tilde{B}_{(2i,i), (2i+2,i+1)} \rightarrow 0
\end{equation}
Now if we let $\overline{\delta_{(i,j)}}$ be the image of $\delta_{(i,j)}$ under $\tilde{A}_{(i,j),(i-1,j+1)} \rightarrow \tilde{B}_{(i,j),(i-1,j+1)}$ then we see that \eqref{eq:TinA'3}, \eqref{eq:decomA'} and \eqref{eq:decomADelta} are compatible in the sense that there is a commutative diagram:
\begin{center}
\begin{tikzpicture}
\matrix(m)[matrix of math nodes,
row sep=3em, column sep=3em,
text height=1.5ex, text depth=0.25ex]
{ 0 & k g'_{3i} & \tilde{A}_{(0,3i), (0,3i+3)} & \tilde{B}_{(0,3i), (0,3i+3)} & 0\\
0 & k g_{(2i,i)} & \tilde{A}_{(2i,i), (2i+2,i+1)} & \tilde{B}_{(2i,i), (2i+2,i+1)} & 0 \\};
\path[->,font=\scriptsize]
(m-1-1) edge (m-1-2)
(m-1-2) edge (m-1-3)
        edge node[left]{$\delta^{-2i} (-) \delta^{2j}$}(m-2-2)
(m-1-3) edge (m-1-4)
        edge node[left]{$\delta^{-2i} (-) \delta^{2j}$}(m-2-3)
(m-1-4) edge (m-1-5)        
        edge node[left]{$\overline{\delta}^{-2i} (-) \overline{\delta}^{2j}$}(m-2-4)
(m-2-4) edge (m-2-5)
(m-2-3) edge (m-2-4)
(m-2-2) edge (m-2-3)
(m-2-1) edge (m-2-2);
\end{tikzpicture}
\end{center}
Hence in order to prove commutativity of \eqref{eq:periodicitycommute} it suffices to prove commutativity of the following two diagrams:
\begin{equation} \label{eq:periodicitycommuteB}
\begin{tikzpicture}
\matrix(m)[matrix of math nodes,
row sep=3em, column sep=6em,
text height=1.5ex, text depth=0.25ex]
{ \tilde{B}_{(2i,i),(2i+2,i+1)} & \tilde{B}_{(0,3i),(0,3i+3)} \\
\tilde{B}_{(2i+2,i+1),(2j+2,j+1)} & \tilde{B}_{(0,3i+3),(0,3i+6)} \\};
\path[->,font=\scriptsize]
(m-1-1) edge node[above]{$\overline{\delta}^{-2i} (-) \overline{\delta}^{2i+2}$}(m-1-2)
        edge node[left]{$\cong$}(m-2-1)
(m-1-2) edge node[right]{$\cong$}(m-2-2)
(m-2-1) edge node[below]{$\overline{\delta}^{-2i-2} (-) \overline{\delta}^{2i+4}$}(m-2-2);
\end{tikzpicture}
\end{equation}
and
\begin{equation}
\label{eq:periodicitycommuteg}
\begin{tikzpicture}
\matrix(m)[matrix of math nodes,
row sep=3em, column sep=6em,
text height=1.5ex, text depth=0.25ex]
{ kg_{(2i,i)} & kg'_{3i} \\
kg_{(2i+2,i+1)} & kg'_{3i+3} \\};
\path[->,font=\scriptsize]
(m-1-1) edge node[above]{$\delta^{-2i} (-) \delta^{2i+2}$}(m-1-2)
        edge node[left]{$\cong$}(m-2-1)
(m-1-2) edge node[right]{$\cong$}(m-2-2)
(m-2-1) edge node[below]{$\delta^{-2i-2} (-) \delta^{2i+4}$}(m-2-2);
\end{tikzpicture}
\end{equation}
We first focus on \eqref{eq:periodicitycommuteB}. The left vertical arrow is given by
\[ \tau^*: \Gamma(Y,\sigma^{*4i}\Lscr \sigma^{*(4i+1)}\Lscr \ldots  \sigma^{*(4i+3)}\Lscr  ( -\tau^{-i}p)) \rightarrow \Gamma(Y,\sigma^{*4i+4}\Lscr  \ldots \sigma^{*(4i+7)}\Lscr ( -\tau^{-i-1}p)) \]
A closer investigation of the 1-periodicity of $\widecheck{B^{\prime (3)}}$ shows that the right vertical arrow in \eqref{eq:periodicitycommuteB} factors as
\begin{equation} \label{eq:factorphi} \tilde{B}_{(0,3i),(0,3i+3)} \overset{\tau^*} \longrightarrow \tilde{B}_{(2,3i+1),(2,3i+4)} \overset{\varphi} \longrightarrow \tilde{B}_{(0,3i+3),(0,3i+6)} \end{equation}
where $\varphi$ is given by multiplication by a nonzero section of
\[ \left(\sigma^{*(6i+4)}\Lscr\sigma^{*(6i+5)}\Lscr (-\tau^{-3i-1}p-\tau^{-3i-2}p) \right)^{-1} \left( \sigma^{*(6i+11)}\Lscr\sigma^{*(6i+12)}\Lscr (-\tau^{-3i-4}p-\tau^{-3i-5}p) \right) \]
As $\overline{\delta_{(m,n)}}$ is a nonzero section of $\sigma^{*(m+2n)}\Lscr(-\tau^{-j}p)$, we can (after inductively changing the $\delta_{(1,a)}$ by a scalar multiple) assume that the morphism
\[ \left(\overline{\delta_{(1,3i+2)}}\right)^{-1}\left( \overline{\delta_{(2,3i+1)}}\right)^{-1} (-) \overline{\delta_{(2,3i+4)}}\ \overline{\delta_{(1,3i+5)}}: \tilde{B}_{(2,3i+1),(2,3i+4)} \rightarrow \tilde{B}_{(0,3i+3),(0,3i+6)}\]
coincides with $\varphi$ for all $i$. Next note that $\tau^*$ maps $k \overline{\delta_{(m,n)}} = \tilde{B}_{(m,n),(m-1,n+1)}$ to $ \tilde{B}_{(m+2,n+1),(m+1,n+2)} = k \overline{\delta_{(m+2,n+1)}}$. Hence (after changing the $\delta_{(1+2b,a+b)}$ and $\delta_{(2+2b,a+b)}$ by a scalar multiple by induction on $b$) we can assume 
\begin{equation} \label{eq:taudelta} \tau^* \overline{\delta_{(m,n)}}=\overline{\delta_{(m+2,n+1)}}
\end{equation}
In particular \eqref{eq:periodicitycommuteB} commutes as it factors as follows:
\begin{center}
\begin{tikzpicture}
\matrix(m)[matrix of math nodes,
row sep=3em, column sep=5em,
text height=1.5ex, text depth=0.25ex]
{ \tilde{B}_{(2i,i),(2i+2,i+1)} & \tilde{B}_{(0,3i),(0,3i+3)}& \\
\tilde{B}_{(2i+2,i+1),(2j+2,j+1)} & \tilde{B}_{(2,3i+1),(2,3i+4)} &\tilde{B}_{(0,3i+3),(0,3i+6)} \\};
\path[->,font=\scriptsize]
(m-1-1) edge node[above]{$\overline{\delta}^{-2i} (-) \overline{\delta}^{2i+2}$}(m-1-2)
        edge node[left]{$\cong$} node[right]{$\tau^*$}(m-2-1)
(m-1-2) edge node[left]{$\tau^*$}(m-2-2)
        edge node[above]{$\cong$}(m-2-3)
(m-2-1) edge node[below]{$\overline{\delta}^{-2i} (-) \overline{\delta}^{2i+2}$}(m-2-2)
(m-2-2) edge node[below]{$\overline{\delta}^{-2} (-) \overline{\delta}^{2} = \varphi$}(m-2-3);
\end{tikzpicture}
\end{center}
Next we focus on commutativity of \eqref{eq:periodicitycommuteg}. The right vertical arrow in \eqref{eq:periodicitycommuteg} can be described as follows: write $g'_{3i}$ as a product of elements in $\check{A'}_{3i,3i+1}=\check{B'}_{3i,3i+1}$, $\check{A'}_{3i+1,3i+2}=\check{B'}_{3i+1,3i+2}$ and $\check{A'}_{3i+2,3i+3}=\check{B'}_{3i+2,3i+3}$, apply $\varphi \circ \tau^*: \check{B'}_{(m,m+1)} \rightarrow \check{B'}_{(m+3,m+4)}$ on each of the 3 factors, then $g'_{3i+3}$ is the product of these 3 new elements. This remark together with \eqref{eq:taudelta} shows that the commutativity of \eqref{eq:periodicitycommuteg} reduces to the following claim:\\
if $g_{(2i,i)} \delta_{(2i,i)}\delta_{(2i-1,i+1)} = a_0 \cdot a_1 \cdot a_2$ for $a_n \in \tilde{A}_{(2i,i+n),(2i,i+n+1)} = \tilde{B}_{(2i,i+n),(2i,i+n+1)}$ then $g_{(2i+2,i+1)} \delta_{(2i+2,i+1)}\delta_{(2i+1,i)} = \tau^*a_0 \cdot \tau^*a_1 \cdot \tau^*a_2$.\\
As each there are embeddings
\begin{equation} \label{eq:embeddingAinA}
\tilde{A}_{(i,j),(m,n)} \hookrightarrow \check{A}_{i+2j,m+2n}
\end{equation}
 it suffices to check the claim in $\check{A}$. For this denote $\underline{x}$ for the image of some $x \in \tilde{A}$ under the embedding \eqref{eq:embeddingAinA}. Consider the equality
\begin{equation} \label{eq:underlineeq} \underline{g_{(2i,i)}} \cdot  \underline{\delta_{(2i,i)}} \cdot \underline{\delta_{(2i-1,i+1)}} = \underline{a_0} \cdot \underline{a_1} \cdot \underline{a_2} \end{equation}
As $\underline{\delta_{(2i,i)}}, \underline{\delta_{(2i-1,i+1)}} ,\underline{a_0}, \underline{a_1} $ and $ \underline{a_2}$ lie in $\check{B}$ the 4-periodicity morphism $\check{A}_{m,n} \rightarrow \check{A}_{m+4,n+4}$ sends them to $\tau^* \underline{\delta_{(2i,i)}}, \tau^* \underline{\delta_{(2i-1,i+1)}} ,\tau^* \underline{a_0} = \underline{\tau^* a_0}, \tau^* \underline{a_1} = \underline{\tau^* a_1}$ and $ \tau^* \underline{a_2} = \underline{\tau^* a_2}$. Moreover by \eqref{eq:taudelta}
\[ \tau^* \underline{\delta_{(i,j)}} = \underline{ \tau^* \overline{\delta_{(i,j)}}} = \underline{\overline{\delta_{(i+2,j+1)}}} = \underline{\delta_{(i+2,j+1)}} \]
Finally notice that by construction $\underline{g_{(2i,i)}}$ is the element in $\check{A}_{4i,4i+4}$ corresponding to the central element $g \in A_4$, as such $\check{A}_{4i,4i+4} \rightarrow \check{A}_{4i+4,4i+8}$ sends it to $\underline{g_{(2i+2,i+1)}}$. In particular the 4-periodicity of $\check{A}$ turns the equality \eqref{eq:underlineeq} into
\[ \underline{g_{(2i+2,i+1)}} \cdot  \underline{\delta_{(2i+2,i+1)}} \cdot \underline{\delta_{(2i+1,i)}} = \underline{\tau^*a_0} \cdot \underline{\tau^*a_1} \cdot \underline{\tau^*a_2}\]
proving our claim and hence showing commutativity of \eqref{eq:periodicitycommuteg} and hence of \eqref{eq:periodicitycommute}. As mentioned above this implies that the inclusion \eqref{eq:TinA'3} induces an inclusion
\[ T \hookrightarrow A^{\prime (3)} \]
Our goal is to understand the image of this inclusion. As $T$ is generated in degree 1, it suffices to understand the image of 
\[ T_1 \cong \tilde{A}_{(0,0),(2,1)} \overset{ \cdot \delta^2}{\longrightarrow} \tilde{A}_{(0,0),(0,3)} \cong A'_3 \]
As was mentioned above this image contains $g'$ and the image of $T_1 \rightarrow A'_3 \rightarrow A'_3/g'$ is the same as the image of
\[ \cdot \overline{\delta_{(2,1)}} \overline{\delta_{(1,2)}}: \tilde{B}_{(0,0),(2,1)} \rightarrow \tilde{B}_{(0,0),(0,3)} \cong B'_3 \]
Following the computations in \cite[\S 8.2]{Presotto} $\overline{\delta_{(2,1)}} \overline{\delta_{(1,2)}}$ is a nonzero global section of
\[ \sigma^{*(-1)}\Lscr \sigma^{*(-2)}\Lscr(-\tau^{-1}p-\tau^{-2}p-q'-p') \]
where $p', q'$ are as in \eqref{eq:defp'q'}.\\
Hence $T$ is the subalgebra of $A^{\prime(3)}$ generated by the elements of $A'_3$ whose image in
\[ B'_3 = \Gamma(Y, \Lscr \sigma^{*(-1)}\Lscr \sigma^{*(-2)}\Lscr(-p-\tau^{-1}p-\tau^{-2}p)) \]
lies in
\[ B'_3 = \Gamma(Y, \Lscr \sigma^{*(-1)}\Lscr \sigma^{*(-2)}\Lscr(-p-\tau^{-1}p-\tau^{-2}p-q'-p')) \]
Hence $T$ is isomorphic to the noncommutative blow up (\cite{RSS2}) $A^{\prime (3)}(p'+q')$ as desired. This finishes the proof of \Cref{thm:commonblowup}.

\bibliographystyle{amsplain}
\bibliography{Notereference}

\providecommand{\bysame}{\leavevmode\hbox to3em{\hrulefill}\thinspace}
\providecommand{\MR}{\relax\ifhmode\unskip\space\fi MR }
\providecommand{\MRhref}[2]{%
  \href{http://www.ams.org/mathscinet-getitem?mr=#1}{#2}
}
\providecommand{\href}[2]{#2}
\begin{thebibliography}{10}

\bibitem{artinschelter}
M.~Artin and W.F. Schelter, \emph{Graded algebras of global dimension 3},
  Adv.Math \textbf{66} (1987), 171--216.

\bibitem{ATV1}
M.~Artin, J.~Tate, and M.~{Van den Bergh}, \emph{Some algebras associated to
  automorphisms of elliptic curves}, The Grothendieck Festschrift (P.~et~al.
  Cartier, ed.), Modern Birkhäuser Classics, vol.~1, Birkhäuser Boston, 1990,
  pp.~33--85.

\bibitem{ATV2}
\bysame, \emph{Modules over regular algebras of dimension 3}, Inventiones
  mathematicae \textbf{106} (1991), no.~1, 335--388.

\bibitem{AV}
M.~Artin and M.~Van~den Bergh, \emph{Twisted homogeneous coordinate rings},
  Journal of Algebra \textbf{133} (1990), no.~2, 249--271.

\bibitem{artinzhang}
M.~Artin and J.J. Zhang, \emph{Noncommutative projective schemes}, Advances in
  Mathematics \textbf{109} (1994), no.~2, 228 -- 287.

\bibitem{chan}
D.~Chan, \emph{Twisted multi-homogeneous coordinate rings}, Journal of Algebra
  \textbf{223} (2000), no.~2, 438 -- 456.

\bibitem{CHTV}
A.~Conca, J.~Herzog, N.V. Trung, and G.~Valla, \emph{Diagonal subalgebras of
  bigraded algebras and embeddings of blow-ups of projective spaces}.

\bibitem{hartshorne}
R.~Hartshorne, \emph{Algebraic geometry}, 8 ed., Graduate Texts in Mathematics,
  Springer-Verslag, 1997.

\bibitem{KSSW}
K.~Kurano, E.~Sato, Anurag K., A.K. Singh, and K.~Watanabe, \emph{Multigraded
  rings, diagonal subalgebras, and rational singularities}, Journal of Algebra
  \textbf{322} (2009), no.~9, 3248 -- 3267.

\bibitem{LowenRS}
W.~Lowen, J.~Ramos~Gonz\'alez, and B.~Shoikhet, \emph{On the tensor product of
  linear sites and grothendieck categories}, arXiv:1607.03608, 2016.

\bibitem{Polishchuk}
A.~Polishchuk, \emph{Noncommutative proj and coherent algebras}, Mathematical
  Research Letters \textbf{12} (2005), no.~1, 63--74.

\bibitem{Presotto}
D.~Presotto, \emph{Symmetric noncommutative birational transformations},
  arXiv:1607.08383, 2016.

\bibitem{PresVdB}
D.~Presotto and M.~Van~den Bergh, \emph{Noncommutative versions of some
  classical birational transformations}, Journal of Noncommutative Geometry
  \textbf{10} (2016), no.~1, 221--244.

\bibitem{RSS2}
D.~Rogalski, S.J. Sierra, and J.T. Stafford, \emph{Noncommutative blowups of
  elliptic algebras}, Algebras and Representation Theory (2014), 1--39.

\bibitem{RSS3}
\bysame, \emph{Ring-theoretic blowing down: I}, arXiv:1603.08128, 2016.

\bibitem{Sierra}
S.~J. Sierra, \emph{{$G$}-algebras, twistings, and equivalences of graded
  categories}, Algebras and Representation Theory \textbf{14} (2011), no.~2,
  377--390.

\bibitem{SimisTV}
A.~Simis, N.V. Trung, and G.~Valla, \emph{The diagonal subalgebra of a blow-up
  algebra}, Journal of Pure and Applied Algebra.

\bibitem{trung}
N.V. Trung, \emph{Diagonal subalgebras and blow-ups of projective spaces},
  Vietnam Journal of Mathematics (2000), no.~28, 1--15.

\bibitem{VdB19}
M.~Van~den Bergh, \emph{Blowing up non-commutative smooth surfaces}, Mem. Amer.
  Math. Soc. \textbf{154} (2001), no.~734.

\end{thebibliography}

\end{document}